\documentclass[12pt]{article}
\usepackage{amssymb, amsmath, amsthm, geometry}
\usepackage{fullpage}

\newcommand{\abs}[1]{\lvert #1 \rvert}

\newtheorem{conjecture}{Conjecture}
\newtheorem{theorem}{Theorem}
\newtheorem{corollary}{Corollary}

\begin{document}

\title{Counterexample to a conjecture of Ziegler concerning
a simple polytope and its dual}

\author{\sc{William Gustafson}}

\maketitle

\begin{abstract}
Problem~4.19 in Ziegler~\cite{Ziegler}
asserts that every simple $3$-dimensional polytope
has the property
that its dual can be constructed
as the convex hull of a subset of the vertices
of the original simple polytope.
In this note we state a higher-dimensional analogue of this
conjecture and provide a family of counterexamples for
dimension $d \geq 3$.
\end{abstract}

\section*{Extended conjecture and counterexamples}

We begin with the following conjecture. We then provide
an infinite number of counterexamples.

\begin{conjecture}
\label{conjecture_higher}
Let $P$ be a simple polytope of dimension greater than or equal to $3$.
Then there exists a subset $S$ of the vertices of $P$ such that
that convex hull of $S$ has the same combinatorial type
as the dual polytope $P^*$.
\end{conjecture}

Ziegler conjectured this result in the
case of $3$-dimensional
simple polytopes; see~\cite[Exercise~4.19]{Ziegler}.
Note that the conjecture is true in dimensions at most $2$
and is immediately true for any $d$-dimensional simplex.
In the case of dimension $3$, 
Andreas Paffenholz, in unpublished work, verified the conjecture
for the truncated tetrahedron, the trucated cube and
the truncated cross-polytope by giving an explicit realization.
It is a straightforward matter to verify 
the conjecture holds for any
$d$-dimensional cube.

We will show that Conjecture~\ref{conjecture_higher} 
is false in all dimensions greater than or equal
to $3$. 	

\begin{theorem}
\label{theorem}
Let $d$ be a positive integer. Let $P$ be a $d'$-dimensional polytope with $d' \ge 2$ and $n$ facets 
such that
every vertex is incident with at most $\left\lceil (n+2d)/2^d\right\rceil-d'$
facets. If $Q$ is the Cartesian product of $P$ with the $d$-dimensional cube
then there is no subset $S$ of the vertices of $Q$
such that the convex hull of $S$ is combinatorially equivalent to the dual polytope
$Q^*$. 
\end{theorem}

	\begin{proof}
		Suppose on the contrary that $S$ is a subset of the vertices
		of the polytope $Q$ satisfying the convex hull of $S$
		is combinatorially equivalent to the dual polytope $Q^*$.
		Observe that the dual polytope $Q^*$
		is combinatorially equivalent to the $d$
		times iterated bipyramid over $P^*$
		and thus has $n+2d$ vertices.
		The vertices of $S$ can then be divided into disjoint subsets
		$T$ and $U$, with $\abs{T} = n$, and $\abs{U} = 2d$
		so that the convex hull of $T$ is combinatorially equivalent to
		$P^*$. The vertices of $U$
		correspond to the $2d$ vertices
		created when taking $d$ iterated bipyramids over $P^*$
		to create $Q^*$.

		Since $Q$ is formed by a Cartesian product,
		the vertices $V(Q)$ of $Q$ can be partitioned into $2^d$
		disjoint sets, say $V(Q) = \bigcup_{i=1}^{2^d}Q_i$.
		Where the convex hull of the vertices in $Q_i$ is combinatorially
		equivalent to a copy of the original polytope $P$
		for $i=1,\dots,2^d$. By the pigeonhole principle, of the $n+2d$
		vertices selected to form the set $S$, there is at least one
		set of the vertex partition of $V(Q)$, say $Q_k$, that contains
		at least $\left\lceil (n+2d)/2^d \right \rceil$ vertices of $S$.

		Let $H$ be a supporting hyperplane for $Q$ at $Q_k$
		which contains no other vertices of $Q$.
		This is also a supporting hyperplane for $S$,
		so the convex hull of the vertices in $S\cap Q_k$ forms a face of 
		$Q^*$, which is at most dimension $d'$.

		By hypothesis each vertex of $P$ is incident with at most
		$\left\lceil (n+2d)/2^d\right\rceil-d'$
		facets. Any vertex of $P$ is in at least $d'$ facets, so
		$\left\lceil (n+2d)/2^d\right\rceil-d' \ge d'$.
		This implies $\left\lceil(n+2d)/2^d\right\rceil \ge 2d'$,
		so $S\cap Q_k$ contains at least $2d'$ vertices.
		We claim at most $d'-1$ of these vertices are
		from $U$. 
		To see this, note that since $S\cap Q_k$
		is the set of vertices of a face of the convex
		hull of $S$,
		it consists of some (possibly 
		empty) set of vertices from $T$, and
		some (possibly empty) set of vertices from $U$.
		These vertices from $U$ lie in general
		position to one another and are additionally in general position
		with respect to $T$. Thus the fact that the convex hull of $S\cap Q_k$ 
		has dimension at most
		$d'$ yields that $\abs{U\cap Q_k} \le d'+1$.

		If $\abs{U\cap Q_k} = d'+1$ then $\abs{T\cap Q_k} \ge d'-1 \ge 1$,
		hence $S\cap Q_k$ contains $d'+2$ vertices in general position,
		so the convex hull has dimension greater then $d'$, a contradiction.
		Similarly if $\abs{U\cap Q_k} = d'$ then $\abs{T\cap Q_k}\ge d'\ge 2$,
		and once again $S\cap Q_k$ would contain $d'+2$ vertices in general
		position.
		Therefore $\abs{U\cap Q_k} \le d'-1$,
		hence $\abs{T\cap Q_k} \ge \left\lceil (n+2d)/2^d\right\rceil -d'+1$.
		Since these vertices are in $T$ and in a face of the convex
		hull of $S$, the convex hull of $T\cap Q_k$ forms
		a proper face of $P^*$. Therefore
		$P^*$ has a facet with at least 
		$\left\lceil (n+2d)/2^d\right\rceil -d'+1$ vertices,
		and thus $P$ has a vertex incident with the same
		number of facets, contrary to assumption.
		Hence there is no such subset $S$ of the vertices
		of $Q$ such that the convex hull of $S$ is combinatorially
		equivalent to $Q^*$.
	\end{proof}

\begin{corollary}
Let $P$ be an $n$-gon with $n \ge 3\cdot 2^d - 2d+1$.
Let $Q$ be the simple polytope formed by taking the Cartesian
product of $P$ with the $d$-dimensional cube. Then there is no subset
$S$ of the vertices of $Q$ where the convex hull of $S$ has the same
combinatorial type as the dual polytope $Q^*$.
\end{corollary}

	\begin{proof}
		Observe that $P$ satisfies the hypothesis of 
		Theorem~\ref{theorem} as $d'=2$
		and $n \ge 3\cdot 2^d - 2d+1$ implies
		$\left\lceil (n+2d)/2^d\right\rceil - d' \ge 2$.
		Each vertex of $P$ is incident with at most 2 facets,
		so indeed by Theorem~\ref{theorem} there is
		no subset $S$ of the vertices of $Q$ where the
		convex hull of $S$ has the same combinatorial type
		as the dual polytope $Q^*$.
	\end{proof}

\begin{conjecture}
	The truncated $d$-simplex, the truncated $d$-cube and the truncated $d$-cross-polytope
	for $d\ge 4$ are families of polytopes where Conjecture~\ref{conjecture_higher} holds.
\end{conjecture}

\section*{Acknowledgements}
The author thanks Richard Ehrenborg and Margaret Readdy
for inspiring conversations
and comments on an earlier draft,
and G\'abor Hetyei for comments.

\newcommand{\bookf}[5]{{\sc #1,} ``#2,'' #3, #4, #5.}

\small

\noindent
{\sc University of Kentucky,
Department of Mathematics,
Lexington, KY 40506.} \hfill\break
{\tt william.gustafson@uky.edu}.


\begin{thebibliography}{}

\bibitem{Ziegler}
\bookf{G{\"u}nter M.\ Ziegler}
      {Lectures on polytopes}
      {Graduate Texts in Mathematics, 152}
      {Springer-Verlag}
      {New York}{1995}

\end{thebibliography}
\end{document}